\documentclass[preprint,12pt]{elsarticle}

\usepackage{color}
\usepackage{amssymb}
\usepackage{amsthm}
\usepackage{amsmath}
\usepackage{epic}
\usepackage{CJK}
\usepackage{setspace}
\usepackage{bm}
\newtheorem{thm}{Theorem}[section]

\newtheorem{lem}[thm]{Lemma}

\newtheorem{defi}[thm]{Definition}

\newtheorem{rem}{Remark}

\journal{~~}

\begin{document}
\begin{spacing}{1.15}
\begin{CJK*}{GBK}{song}
\begin{frontmatter}
\title{\textbf{A tensor's spectral bound on the clique number}}

\author[label a]{Chunmeng Liu\corref{cor}}\ead{liuchunmeng214@nenu.edu.cn/liuchunmeng0214@126.com}
\author[label b]{Changjiang Bu}\ead{buchangjiang@hrbeu.edu.cn}
\cortext[cor]{Corresponding author}

\address{
\address[label a]{Academy for Advanced Interdisciplinary Studies, Northeast Normal University, Changchun 130024, PR China}
\address[label b]{College of Mathematical Sciences, Harbin Engineering University, Harbin 150001, PR China}
}

\begin{abstract}
In this paper, we study the spectral radius of the clique tensor $\mathcal{A}(G)$ associated with a graph $G$. 
This tensor is a higher-order extensions of the adjacency matrix of $G$. 
A lower bound of the clique number is given via the spectral radius of $\mathcal{A}(G)$. 
It is an extension of Nikiforov's spectral bound and tighter than the bound of Nikiforov in some classes of graphs. 
Furthermore, we obtain a spectral version of the Erd\H{o}s-Simonovits stability theorem for clique tensors based on this bound.
\end{abstract}

\begin{keyword}
Tensor; Spectral radius; Clique number; Stability theorem
\\
\emph{AMS classification:} 05C50, 05C35
\end{keyword}
\end{frontmatter}

\section{Introduction}

The graphs considered throughout the paper are all simple. 
A \emph{clique} is a subset of vertices within a graph that forms a complete subgraph. 
The \emph{clique number} of a graph is defined as the size of its largest clique.
A clique consisting of $t$ vertices is referred to as a \emph{$t$-clique}.
The collection of all $t$-cliques in a graph $G$ is represented by $C_{t}(G)$, and the count of $t$-cliques in $G$ is denoted by $|C_{t}(G)|$. 

\subsection{Spectral radius and the clique number}

For a graph $G$, let $\rho(G)$ denote the largest eigenvalue of its adjacency matrix, which is also called the \emph{spectral radius} of $G$. 
The research on the matrix's spectral bounds of clique numbers has yielded numerous significant results. 
For spectral bounds of the adjacency matrix, refer to \cite{Wilf_1986,Bollobs_2007,Nikiforov_2009_2}. 
For the Laplacian matrix's spectral bounds, see \cite{Lu_2007} and for the signless Laplacian matrix's spectral bounds, see \cite{He_2013}. 
In this paper, we obtain a bound for the clique number via the tensor's spectra.

An order $m$ dimension $n$ complex tensor $\mathcal{A}=(a_{i_{1}i_{2}\cdots i_{m}})$ is a multidimensional array comprising $n^{m}$ entries, where each index $i_{j}=1,2,\cdots,n$ for $j=1,2,\cdots,n$. 
In 2005, Qi \cite{Qi_2005} and Lim \cite{Lim_2005} independently introduced the concept of eigenvalues for tensors. 
Recently, the authors \cite{Liu_2023} proposed the \emph{$t$-clique tensor} of a graph. 
Specifically, the $2$-clique tensor is the adjacency matrix.

\begin{defi}\textup{\cite{Liu_2023}}
Let $G$ be an $n$-vertex graph. 
An order $t$ and dimension $n$ tensor $\mathcal{A}(G)=(a_{i_{1}i_{2}\cdots i_{t}})$ is called the $t$-clique tensor of $G$, if
\begin{align*}
a_{i_{1}i_{2}\cdots i_{t}}=
\begin{cases}
\frac{1}{(t-1)!}, &\{i_{1},\cdots,i_{t}\}\in C_{t}(G).\\
0, &otherwise.
\end{cases}
\end{align*}
\end{defi}
The maximum modulus of all eigenvalues of the $t$-clique tensor of graph $G$ is referred to as the \emph{$t$-clique spectral radius} of $G$, denoted by $\rho_t(G)$. 
When $t=2$, it simplifies to the spectral radius of the adjacency matrix of $G$.

In this paper, we give a bound for the clique number in terms of $t$-clique spectral radius. 
For two nonnegative integers $z$ and $s$, let $\binom{z}{s}=\begin{cases}
\frac{z!}{s!(z-s)!}, & z\geq s.\\
0, & z<s.
\end{cases}$
\begin{thm}\label{thm main}
For a graph $G$, let $\omega$ be the clique number of $G$. 
Then
\begin{align*}
\rho_{t}(G)\leq\frac{t}{\omega}\binom{\omega}{t}^{\frac{1}{t}}|C_{t}(G)|^{\frac{t-1}{t}}.
\end{align*}
Moreover, if $G$ is a complete regular $\omega$-partite graph for $\omega\geq t\geq2$, then the equality is achieved in the above inequality.
\end{thm}
In the case of $t=2$, Theorem \ref{thm main} is a result of Nikiforov \cite{Nikiforov_2002}. 
In Section 3, we will show that our bound tighter than Nikiforov's bound in some classes of graphs and Theorem \ref{thm main} implies that a Tur\'{a}n-type result of Erd\H{o}s \cite{Erdos_1962}. 

\subsection{High-order spectral Tur\'{a}n problem}
Let $\mathcal{F}$ denote a family of $n$-vertex graphs that contain no isolated vertices.
A graph is called $\mathcal{F}$-free if it does not contain any graph from $\mathcal{F}$ as a subgraph.
We denote by $ex(n,\mathcal{F})$ the maximum number of edges in an $\mathcal{F}$-free graph on $n$ vertices.
The study of $ex(n,\mathcal{F})$ is an important topic in extremal graph theory. 
There have been many classical results, such as Mantel's theorem \cite{Mantel_1907}, Tur\'{a}n's theorem \cite{Turan_1941}, Erd\H{o}s-Stone-Simonovits theorem \cite{Erdos_1946,Erdos_1966} and \cite{Furedi_2013} for a survey.

We write $ex_{\rho}(n,\mathcal{F})$ for the maximum spectral radius over all $\mathcal{F}$-free graphs on $n$ vertices.
The problem of determining $ex_{\rho}(n,\mathcal{F})$ is referred to as the spectral version of the Tur\'{a}n problem.
Many classical Tur\'{a}n-type results have been extended to their spectral versions, for example, spectral Mantel's theorem \cite{Nosal_1970}, spectral Tur\'{a}n's theorem \cite{Nikiforov_2007}, spectral Erd\H{o}s-Stone-Bollob\'{a}s theorem \cite{Nikiforov_2009} and for some further results see \cite{Zhai_2022,Cioaba_2023,Peng_2024,Nikiforov_2011}.

In this paper, we consider the high-order spectral Tur\'{a}n problem, which involves determining the maximum $t$-clique spectral radius over all $\mathcal{F}$-free graphs on $n$ vertices. 
In \cite{Liu_2023}, the authors obtained the maximum $r$-clique spectral radius over all $K_{r+1}$-free graphs on $n$ vertices, which is the high-order spectral version of Mantel's theorem. 
In \cite{Liu_2024}, the authors provided an upper bound for $3$-clique spectral radius of a $B_{k}$-free and $K_{2,l}$-free graph on $n$ vertices. 

This paper concentrates on the stability result in the Tur\'{a}n-type problem.
A structural stability theorem obtained by Erd\H{o}s \cite{Erdos_1968} and Simonovits \cite{Simonovits_1968}, which is called Erd\H{o}s-Simonovits stability theorem.
\begin{thm}\textup{\cite{Erdos_1968,Simonovits_1968}}\label{thm Sim}
Let $H$ be a graph with the chromatic number $\chi(H) = r + 1 > 2$. 
For every $\epsilon > 0$, there exist $\delta>0$ and $n_{0}$ such that if $G$ is an $H$-free graph on $n \geq n_{0}$ vertices and $|C_{2}(G)|\geq\left(1-\frac{1}{r}-\delta\right)\frac{n^{2}}{2}$, then $G$ can be obtained from $T_{r} (n)$ by adding and deleting at most 
$\epsilon n^{2}$ edges.
\end{thm}

In 2009, Nikiforov \cite{Nikiforov_2009_1} proved a spectral version of Erd\H{o}s-Simonovits stability theorem.
\begin{thm}\textup{\cite{Nikiforov_2009_1}}\label{thm Nikiforov}
Let $H$ be a graph with $\chi(H) = r + 1 > 2$. 
For every $\epsilon > 0$, there exist $\delta>0$ and $n_{0}$ such that if $G$ is an $H$-free graph on $n \geq n_{0}$ vertices and $\rho(G)\geq\left(1-\frac{1}{r}-\delta\right)n$, then $G$ can be obtained from $T_{r} (n)$ by adding and deleting at most  
$\epsilon n^{2}$ edges.
\end{thm}

In this paper, we extend Theorem \ref{thm Nikiforov} to the $t$-clique spectral version based on Theorem \ref{thm main}.

\begin{thm}\label{thm1}
Let $H$ be a graph with $\chi(H) = r + 1 > t \geq 2$. 
For every $\epsilon > 0$, there exist $\delta>0$ and $n_{0}$ such that if $G$ is an $H$-free graph on $n \geq n_{0}$ vertices and $\rho_{t}(G)\geq\left(\binom{r-1}{t-1}\left(\frac{1}{r}\right)^{t-1}-\delta\right)n^{t-1}$, then $G$ can be obtained from $T_{r} (n)$ by adding and deleting at most  
$\epsilon n^{2}$ edges.
\end{thm}

The remainder of this paper is organized as follows: 
In Section 2, we introduce some definitions and lemmas required for the proofs. 
In Section 3, we present proofs of the conclusions of this paper along with some remarks.

\section{Preliminaries}

For an order $m$ dimensional $n$ complex tensor $\mathcal{A}=(a_{i_{1}i_{2}\cdots i_{m}})$, if there exist a complex number $\lambda$ and a nonzero complex vector $x=(x_{1},\ldots,x_{n})^{T}$ satisfy
\begin{align}\label{equ1}
\lambda x_{i}^{m-1}=\sum_{i_{2},\ldots,i_{m}=1}^{n}a_{ii_{2}\cdots i_{m}}x_{i_{2}}\cdots x_{i_{m}} \ (i=1,2,\cdots,n),
\end{align}
then $\lambda$ is called an \emph{eigenvalue} of $\mathcal{A}$ and $x$ is called an \emph{eigenvector} of $\mathcal{A}$ associated with $\lambda$ (refer to \cite{Qi_2005} for further details). 
A tensor $\mathcal{A}$ is termed symmetric if its entries remain invariant under any permutation of their indices. 
Furthermore, if all entries of a tensor $\mathcal{A}$ are nonnegative, then $\mathcal{A}$ is referred to as a nonnegative tensor.

\begin{lem}\textup{\cite{Qi_2013}}\label{lem Qi}
Let $\mathcal{A}=(a_{i_{1}i_{2}\cdots i_{m}})$ be an order $m$ dimension $n$ symmetric nonnegative tensor. 
The spectral radius of $\mathcal{A}$ is equal to
\begin{align*}
\max\left\{\sum_{i_{1},i_{2},\cdots,i_{m}=1}^{n}a_{i_{1}\cdots i_{m}}x_{i_{1}}\cdots x_{i_{m}}:\sum_{i=1}^{n}x_{i}^{m}=1,(x_{1},x_{2},\cdots,x_{n})^{T}\in \mathbb{R}_{+}^{n}\right\},
\end{align*}
where $\mathbb{R}_{+}^{n}$ denotes the set of all $n$-dimensional vectors with nonnegative components.
\end{lem}

\begin{lem}\textup{\cite{Liu_2023}}\label{lem inequation of clique}
For a graph $G$ on $n$ vertices, we have
\begin{align}\label{equ p-spectral radio 1}
|C_{t}(G)|\leq\frac{n}{t}\rho_{t}(G).
\end{align}
Furthermore, if the number of $t$-cliques containing vertex $i$ is the same for all $i\in V(G)$, then equality holds in \textup{(\ref{equ p-spectral radio 1})}.
\end{lem}

\begin{lem} \label{Erdos} \textup{\cite{Erdos}}
Let $\epsilon$ be an arbitrary positive number, and let $G$ be an $H$-free graph on $n$ vertices. 
There exists a positive number  $n_{0}$ such that for $n > n_{0}$, one can remove less than $\epsilon n^{2}$ edges from $G$ so that the remaining graph is $K_{r+1}$-free, where $r+1=\chi(H)$.
\end{lem}

\begin{lem}\textup{\cite{Sos_1982}}\label{lem Sos}
Let $G$ be an $n$-vertex $K_{r+1}$-free graph and let $r\geq t\geq s\geq1$. 
Then
\begin{align*}
\left(\frac{|C_{t}(G)|}{\binom{r}{t}}\right)^{\frac{1}{t}}\leq\left(\frac{|C_{s}(G)|}{\binom{r}{s}}\right)^{\frac{1}{s}}.
\end{align*}
\end{lem}

In our proofs, we will use H\"{o}lder's inequality, which is for two nonnegative vectors $x=(x_{1},\ldots,x_{n})^{T}$ and $y=(y_{1},\ldots,y_{n})^{T}$. 
If the positive numbers $p$ and $q$ satisfy $\frac{1}{p}+\frac{1}{q}=1$, then 
\begin{align*}
\sum_{i=1}^{n}x_{i}y_{i}\leq\left(\sum_{i=1}^{n}x_{i}^{p}\right)^{\frac{1}{p}}\left(\sum_{i=1}^{n}y_{i}^{q}\right)^{\frac{1}{q}}
\end{align*}
with equality holding if and only if $x$ and $y$ are proportional.
And Maclaurin's inequality, which states that for a nonnegative vector $x=(x_{1},\ldots,x_{n})^{T}$, the following holds for each $k\in\{1,2,\ldots,n\}$, 
\begin{align*}
\frac{x_{1}+\cdots+x_{n}}{n}\geq\left(\frac{\sum\limits_{1\leq i_{1}<\cdots<i_{k}\leq n}x_{i_{1}}\cdots x_{i_{k}}}{\binom{n}{k}}\right)^{\frac{1}{k}},
\end{align*}
the equality holds if and only if $x_{1}=\cdots=x_{n}$.

\section{The proof of Theorem}

To prove Theorem \ref{thm main}, we give the following conclusion, which is a high-order extension of the Motzkin-Straus theorem \cite{St_1965}. 
For a graph $G$ on $n$ vertices, let 
\begin{align*}
\mu_{t}(G)=\max\left\{\sum_{\{i_{1},i_{2},\cdots, i_{t}\}\in C_{t}(G)}x_{i_{1}}x_{i_{2}}\cdots x_{i_{t}}:\sum_{i=1}^{n}x_{i}=1,(x_{1},\cdots,x_{n})^{T}\in\mathbb{R}_{+}^{n}\right\},
\end{align*}
where $\mathbb{R}_{+}^{n}$ is the set of all $n$-dimensional nonnegative vectors.
\begin{lem}\label{lem1}
Let $G$ be a graph with clique number $\omega$ and let $t\leq \omega$. 
Then
\begin{align*}
\mu_{t}(G)=\binom{\omega}{t}\omega^{-t}.
\end{align*}
\end{lem}
\begin{proof}
Let $G$ be a graph consisting of an $\omega$-clique and $n-\omega$ isolated vertices. 
Let $x_{1}=x_{2}=\cdots=x_{\omega}=\frac{1}{\omega}$ and $x_{\omega+1}=\cdots=x_{n}=0$. 
Then
\begin{align*}
\mu_{t}(G)\geq \binom{\omega}{t}\omega^{-t}.
\end{align*}
Suppose that $x=(x_{1},x_{2},\cdots,x_{n})^{T}$  is a vector chosen such that the expression 
\begin{align*}
\sum_{\{i_{1},i_{2},\cdots, i_{t}\}\in C_{t}(G)}x_{i_{1}}x_{i_{2}}\cdots x_{i_{t}}
\end{align*}
achieves its maximum value, and that $x$ contains the minimal number of nonzero entries among all such vectors. 
Let the vector $x$ have exactly $s$ nonzero entries. 
To prove that $s\leq \omega$, we proceed by contradiction. 
Assuming that $s>\omega$, there exist two nonzero entries $x_{i}$ and $x_{j}$ in $x$ such that the vertices $i$ and $j$ are not adjacent in $G$. 
Otherwise, the subgraph induced by the nonzero entries of $x$ would form a clique larger than $\omega$. 
Let $S_{G}(v,x)=\frac{\partial\left(\sum\nolimits_{\{i_{1},\cdots, i_{t}\}\in C_{t}(G)}x_{i_{1}}\cdots x_{i_{t}}\right)}{\partial x_{v}}$ for $v=1,2,\cdots,n$.
Therefore, we have
\begin{align*}
\sum_{\{i_{1},\cdots, i_{t}\}\in C_{t}(G)}x_{i_{1}}\cdots x_{i_{t}}
=x_{i}S_{G}(i,x)+x_{j}S_{G}(j,x)+M,
\end{align*}
where the monomial in $M$ contains neither $x_{i}$ nor $x_{j}$. 
Suppose that $S_{G}(i,x)\geq S_{G}(j,x)$, the proof for the scenario where $S_{G}(i,x)< S_{G}(j,x)$ would follow a similar approach. 
Constructing the vector $y=(y_{1},y_{2},\cdots,y_{n})^{T}$, where
\begin{align*}
y_{i}=x_{i}+x_{j}, \ y_{j}=0, \ y_{m}=x_{m} \ (m\in\{1,2,\cdots,n\} \setminus \{i,j\}).
\end{align*}
Obviously, we have $y_{1}+y_{2}+\cdots+y_{n}=1$ and
\begin{align*}
\sum_{\{i_{1},\cdots, i_{t}\}\in C_{t}(G)}y_{i_{1}}\cdots y_{i_{t}}-\sum_{\{i_{1},\cdots, i_{t}\}\in C_{t}(G)}x_{i_{1}}\cdots x_{i_{t}}
=x_{j}\left(S_{G}(i,x)-S_{G}(j,x)\right).
\end{align*}
By the assumption $x_{j}>0$ and $S_{G}(i,x)\geq S_{G}(j,x)$, we obtain
\begin{align}\label{equ2}
\sum_{\{i_{1},\cdots, i_{t}\}\in C_{t}(G)}y_{i_{1}}\cdots y_{i_{t}}-\sum_{\{i_{1},\cdots, i_{t}\}\in C_{t}(G)}x_{i_{1}}\cdots x_{i_{t}}\geq 0.
\end{align}
If $\sum_{\{i_{1},\cdots, i_{t}\}\in C_{t}(G)}y_{i_{1}}\cdots y_{i_{t}}-\sum_{\{i_{1},\cdots, i_{t}\}\in C_{t}(G)}x_{i_{1}}\cdots x_{i_{t}}> 0$, this would contradict the choice of $x$ as the vector that maximizes the expression $\sum_{\{i_{1},\cdots, i_{t}\}\in C_{t}(G)}x_{i_{1}}\cdots x_{i_{t}}$. 
And if the equality holds in (\ref{equ2}), the vector $y$ has $s-1$ nonzero entries, this would contradict the assumption that vector $x$ has the least number of nonzero entries among all vectors that achieve the maximum value of the expression. 
Thus, we have $s\leq \omega$. 
Suppose that $x_{1},x_{2},\cdots,x_{s}$ are all nonzero entries in $x$. 
By Maclaurin's inequality, we obtain
\begin{align*}
\mu_{t}(G)
&=\sum_{\{i_{1},i_{2},\cdots, i_{t}\}\in C_{t}(G)}x_{i_{1}}x_{i_{2}}\cdots x_{i_{t}}
\leq \sum_{1\leq i_{1}<i_{2}<\cdots<i_{t}\leq s}x_{i_{1}}x_{i_{2}}\cdots x_{i_{t}}
\leq \binom{s}{t}\left(\frac{1}{s}\sum_{i=1}^{s}x_{i}\right)^{t}\\
&= \binom{s}{t}s^{-t}
=\frac{1}{t!}\left(1-\frac{1}{s}\right)\left(1-\frac{2}{s}\right)\cdots\left(1-\frac{t-1}{s}\right)\\
&\leq \frac{1}{t!}\left(1-\frac{1}{\omega}\right)\left(1-\frac{2}{\omega}\right)\cdots\left(1-\frac{t-1}{\omega}\right)
= \binom{\omega}{t}\omega^{-t}.
\end{align*}
The proof is complete.
\end{proof}

Next, we present the proof of Theorem \ref{thm main}.

\begin{proof}[Proof of Theorem \ref{thm main}]
For a graph $G$, let $\mathcal{A}(G)=(a_{i_{1}i_{2}\cdots i_{t}})$ be the $t$-clique tensor of $G$.
Let $x=(x_{1},x_{2},\cdots,x_{n})^{T}$ be an nonnegative eigenvector corresponding to $\rho_{t}(G)$ with $x_{1}^{t}+\cdots+x_{n}^{t}=1$. 
Then
\begin{align}\label{equ3}
\rho_{t}(G)=\frac{\sum_{i_{1},\cdots,i_{t}=1}^{n}a_{i_{1}\cdots i_{t}}x_{i_{1}}\cdots x_{i_{t}}}{x_{1}^{t}+\cdots+x_{n}^{t}}=t\sum_{\{i_{1},i_{2},\cdots, i_{t}\}\in C_{t}(G)}x_{i_{1}}x_{i_{2}}\cdots x_{i_{t}}.
\end{align}
By H\"{o}lder's inequality, we have
\begin{align*}
\rho_{t}(G)\leq t\left(\sum_{\{i_{1},i_{2},\cdots, i_{t}\}\in C_{t}(G)} 1^{\frac{t}{t-1}}\right)^{\frac{t-1}{t}}\left(\sum_{\{i_{1},i_{2},\cdots, i_{t}\}\in C_{t}(G)}x_{i_{1}}^{t}x_{i_{2}}^{t}\cdots x_{i_{t}}^{t}\right)^{\frac{1}{t}}.
\end{align*}
Since $x_{1}^{t}+\cdots+x_{n}^{t}=1$, by Lemma \ref{lem1}, we obtain
\begin{align*}
\rho_{t}(G)\leq t|C_{t}(G)|^{\frac{t-1}{t}}\left(\binom{\omega}{t}\omega^{-t}\right)^{\frac{1}{t}}
=\frac{t}{\omega}\binom{\omega}{t}^{\frac{1}{t}}|C_{t}(G)|^{\frac{t-1}{t}}.
\end{align*}

Let $G$ be a complete regular $\omega$-partite graph for $\omega\geq t\geq2$. 
The number of $t$-cliques in $G$ is $|C_{t}(G)|=\omega^{t}\binom{\omega}{t}$, then
\begin{align*}
\frac{t}{\omega}\binom{\omega}{t}^{\frac{1}{t}}|C_{t}(G)|^{\frac{t-1}{t}}=\omega^{t-1}\binom{\omega-1}{t-1}.
\end{align*}
And each vertex is contained in $\omega^{t-1}\binom{\omega-1}{t-1}$ $t$-cliques. 
Thus, we have
\begin{align*}
\rho_{t}(G)=\omega^{t-1}\binom{\omega-1}{t-1}.
\end{align*}
\end{proof}

\begin{rem}
In \cite{Nikiforov_2002}, Nikiforov prove that 
\begin{align}\label{equ niki_2002}
2|C_{2}(G)|\frac{\omega-1}{\omega}\geq\rho^{2}(G).
\end{align}
Let $F$ be a unicyclic graph with girth $3$. 
The graph $F$ contains a subgraph isomorphic to $K_{3}$, then $\rho(F)\geq 2$. 
By (\ref{equ niki_2002}), we have the following inequality for the clique number $\omega(F)$ of $F$, 
\begin{align}\label{inequ niki}
\omega(F)\geq 1+\frac{2}{|C_{2}(F)|-2}.
\end{align}
For the graph $F$, the $3$-clique spectral radius $\rho_{3}(F)=1$.
By Theorem \ref{thm main}, we get
\begin{align}\label{inequ our}
\omega(F)\geq 3.
\end{align}
Obviously, the inequality (\ref{inequ our}) is tighter than (\ref{inequ niki}) when $|C_{2}(F)|$ is larger than $3$.

In addition, for a $K_{r+1}$-free graph $F^{\prime}$ on $n$ vertices, by Theorem \ref{thm main} and Lemma \ref{lem inequation of clique}, we have
\begin{align*}
|C_{t}(F^{\prime})|\leq \frac{n}{t}\frac{t}{r}\binom{r}{t}^{\frac{1}{t}}|C_{t}(F^{\prime})|^{\frac{t-1}{t}}.
\end{align*}
Then
\begin{align*}
|C_{t}(F^{\prime})|\leq\frac{n^{t}}{r^{t}}\binom{r}{t}.
\end{align*}
Theorem \ref{thm main} and Lemma \ref{lem inequation of clique} imply that an upper bound on the number of $t$-cliques in a $K_{r+1}$-free graph, a conclusion established in \cite{Erdos_1962}.
\end{rem}

\begin{proof}[Proof of Theorem \ref{thm1}]
For a graph $H$ with $\chi(H)=r+1\geq3$, let $G$ be an $H$-free graph on $n$ vertices. 
Let $x=(x_{1},x_{2},\cdots,x_{n})^{T}$ be a nonnegative eigenvector corresponding to $\rho_{t}(G)$ with $x_{1}^{t}+\cdots+x_{n}^{t}=1$.
By (\ref{equ3}), we have
\begin{align*}
\rho_{t}(G)=t\sum_{\{i_{1},i_{2},\cdots, i_{t}\}\in C_{t}(G)}x_{i_{1}}x_{i_{2}}\cdots x_{i_{t}}.
\end{align*}
For an arbitrary positive number $\epsilon>0$, let $\epsilon^{\prime}\in\left(0,\min\left\{\frac{\epsilon}{2},\left((t-1)!\epsilon\right)^{\frac{t}{t-1}}\right\}\right]$. 
By Lemma \ref{Erdos}, for a sufficiently large $n$, we can get a $K_{r+1}$-free graph $G^{\prime}$ from $G$ by removing less than $\epsilon^{\prime} n^{2}$ edges.
This process implies that the number of $t$-cliques removed from $G$ is at most $\epsilon^{\prime} n^{2}\binom{n-2}{t-2}$. 
Let $\hat{G}$ be an $n$-vertex graph with $C_{t}(\hat{G})=C_{t}(G)\setminus C_{t}(G^{\prime})$. 
Then
\begin{align*}
t\sum_{\{i_{1},\cdots, i_{t}\}\in C_{t}(G)}x_{i_{1}}\cdots x_{i_{t}}
=t\sum_{\{i_{1},\cdots, i_{t}\}\in C_{t}(G^{\prime})}x_{i_{1}}\cdots x_{i_{t}}
+t\sum_{\{i_{1},\cdots, i_{t}\}\in C_{t}(\hat{G})}x_{i_{1}}\cdots x_{i_{t}}.
\end{align*}
By Lemma \ref{lem Qi}, we obtain
\begin{align*}
\rho_{t}(G)\leq\rho_{t}(G^{\prime})+\rho_{t}(\hat{G}).
\end{align*}
For the graph $\hat{G}$, since $|C_{t}(\hat{G})|\leq\epsilon^{\prime} n^{2}\binom{n-2}{t-2}$ and by Theorem \ref{thm main}, 
\begin{align*}
\rho_{t}(\hat{G})\leq\frac{t}{n}\binom{n}{t}^{\frac{1}{t}}|C_{t}(\hat{G})|^{\frac{t-1}{t}}
\leq \left(\frac{t^{t}}{n^{t}}\frac{n^{t}}{t!}\right)^{\frac{1}{t}}\left(\epsilon^{\prime}\frac{n^{t}}{(t-2)!}\right)^{\frac{t-1}{t}}
\leq\frac{(\epsilon^{\prime})^{\frac{t-1}{t}}}{(t-1)!}n^{t-1}
\leq\epsilon n^{t-1}.
\end{align*}
Thus, we get
\begin{align*}
\rho_{t}(G^{\prime})\geq \rho_{t}(G)-\epsilon n^{t-1}.
\end{align*}
For a $K_{r+1}$-free graph $F$, by Theorem \ref{thm Sim}, for a positive number $\epsilon^{\prime\prime}\in(0,\frac{\epsilon}{2}]$, there exist $\delta^{\prime}>0$ and a sufficiently large natural number $n$ such that if $F$ is a graph on $n$ vertices and satisfies $|C_{2}(F)|\geq\left(1-\frac{1}{r}-\delta^{\prime}\right)\frac{n^{2}}{2}$, then $F$ can be obtained from $T_{r} (n)$ by adding and deleting at most 
$\epsilon^{\prime\prime} n^{2}$ edges.
Let $\delta=\left(\frac{\binom{r-1}{t-1}}{r(r-1)}\delta^{\prime}\right)^{\frac{t-2}{2}}-\epsilon$. 
Assume $n$ is a natural number that is sufficiently large such that 
\begin{align*}
\rho_{t}(G^{\prime})\geq\left(\binom{r-1}{t-1}\left(\frac{1}{r}\right)^{t-1}-\delta\right)n^{t-1}. 
\end{align*}
Then
\begin{align*}
\rho_{t}(G^{\prime})\geq\left(\binom{r-1}{t-1}\left(\frac{1}{r}\right)^{t-1}-\delta\right)n^{t-1}-\epsilon n^{t-1}=\left(\binom{r-1}{t-1}\left(\frac{1}{r}\right)^{t-1}-\delta-\epsilon\right)n^{t-1}.
\end{align*}
The graph $G^{\prime}$ is $K_{r+1}$-free, by Theorem \ref{thm main}, we obtain
\begin{align*}
\frac{t}{r}\binom{r}{t}^{\frac{1}{t}}|C_{t}(G^{\prime})|^{\frac{t-1}{t}}
\geq\rho_{t}(G^{\prime})\geq\left(\binom{r-1}{t-1}\left(\frac{1}{r}\right)^{t-1}-\delta-\epsilon\right)n^{t-1}.
\end{align*}
By Lemma \ref{lem Sos}, we have
\begin{align*}
\frac{t}{r}\binom{r}{t}\left(\frac{|C_{2}(G^{\prime})|}{\binom{r}{2}}\right)^{\frac{t-1}{2}}
\geq\frac{t}{r}\binom{r}{t}^{\frac{1}{t}}|C_{t}(G^{\prime})|^{\frac{t-1}{t}}
\geq\left(\binom{r-1}{t-1}\left(\frac{1}{r}\right)^{t-1}-\delta-\epsilon\right)n^{t-1}.
\end{align*}
Then
\begin{align*}
|C_{2}(G^{\prime})|
&\geq\left(\left(\frac{1}{r}\right)^{t-1}\binom{r}{2}^{\frac{t-1}{2}}
-\frac{r}{t}\binom{r}{t}^{-1}\binom{r}{2}^{\frac{t-1}{2}}(\delta+\epsilon)\right)^{\frac{2}{t-1}}n^{2}\\
&\geq\left(\left(\frac{1}{r}\right)^{2}\binom{r}{2}-\binom{r-1}{t-1}^{-1}\binom{r}{2}(\delta+\epsilon)^{\frac{2}{t-1}}\right)n^{2}\\
&=\left(1-\frac{1}{r}-\binom{r-1}{t-1}^{-1}r(r-1)(\delta+\epsilon)^{\frac{2}{t-1}}\right)\frac{n^{2}}{2}
=\left(1-\frac{1}{r}-\delta^{\prime}\right)\frac{n^{2}}{2}.
\end{align*}
By Theorem \ref{thm Sim}, the graph $G^{\prime}$ can be obtained from $T_{r} (n)$ by adding and deleting at most $\epsilon^{\prime\prime} n^{2}$ edges. 
Since we removed $\epsilon^{\prime} n^{2}$ edges from $G$ to $G^{\prime}$, the graph $G$ can be derived from $T_{r} (n)$ by adding and deleting at most $\epsilon^{\prime} n^{2}+\epsilon^{\prime\prime} n^{2}\leq\epsilon n^{2}$ edges. 
\end{proof}

\section*{Acknowledgement}

This work is supported by the National Natural Science Foundation of China (No. 12071097, 12371344), the Natural Science Foundation for The Excellent Youth Scholars of the Heilongjiang Province (No. YQ2022A002) and the Fundamental Research Funds for the Central Universities.

\end{CJK*}
\end{spacing}
\end{document}